\documentclass[11pt]{amsart}

\usepackage[%pagebackref,
hypertexnames=false, colorlinks, citecolor=red, linkcolor=red]{hyperref} %,hypertexnames=false,colorlinks,[pagebackref]
\hypersetup{bookmarksdepth=3}

\usepackage{amssymb}
	\normalbaselines						  %

\newcommand{\R}{{\mathbb  R}}

\newcommand{\Z}{{\mathbb  Z}}

\newcommand{\N}{{\mathbb  N}}

\newcommand{\C}{{\mathbb  C}}

\newcommand{\ID}{{\mathbf{1}}}

\newcommand{\OID}{{\mathbf{I}}}

\newcommand{\fdot}{\,\cdot\,}

\makeatletter
\def\Ddots{\mathinner{\mkern1mu\raise\p@
\vbox{\kern7\p@\hbox{.}}\mkern2mu
\raise4\p@\hbox{.}\mkern2mu\raise7\p@\hbox{.}\mkern1mu}}
\makeatother

\newcommand{\cH}{\mathcal{H}}

\newcommand{\f}{\varphi}

\newcommand{\e}{\varepsilon}

\newcommand{\p}{\mathbb{P}}

\DeclareMathOperator{\spa}{span}

\DeclareMathOperator{\clos}{clos}
\DeclareMathOperator{\Mod}{mod}
\DeclareMathOperator{\supp}{supp}
\DeclareMathOperator{\esupp}{ess-supp}
\DeclareMathOperator{\dist}{dist}

\newcommand{\ci}[1]{_{ {}_{\scriptstyle #1}}}
\newcommand{\ti}[1]{_{\scriptstyle \text{\rm #1}}}
\catcode`@=11
\chardef\mathlig@atcode\count255

\def\actively#1#2{\begingroup\uccode`\~=`#2\relax\uppercase{\endgroup#1~}}
\def\mathlig@gobble{\afterassignment\mathlig@next@cmd\let\mathlig@next= }
\def\mathlig@delim{\mathlig@delim}
\def\mathlig@defcs#1{\expandafter\def\csname#1\endcsname}
\def\mathlig@let@cs#1#2{\expandafter\let\expandafter#1\csname#2\endcsname}
\def\mathlig@appendcs#1#2{\expandafter\edef\csname#1\endcsname{\csname#1\endcsname#2}}

\def\mathlig#1#2{\mathlig@checklig#1\mathlig@end\mathlig@defcs{mathlig@back@#1}{#2}\ignorespaces}

\def\mathlig@checklig#1#2\mathlig@end{%
 \expandafter\ifx\csname mathlig@forw@#1\endcsname\relax
 \expandafter\mathchardef\csname mathlig@back@#1\endcsname=\mathcode`#1%
 \mathcode`#1"8000\actively\def#1{\csname mathlig@look@#1\endcsname}%
 \mathlig@dolig#1\mathlig@delim
\fi
\mathlig@checksuffix#1#2\mathlig@end
}

\def\mathlig@checksuffix#1#2\mathlig@end{%
\ifx\mathlig@delim#2\mathlig@delim\relax\else\mathlig@checksuffix@{#1}#2\mathlig@end\fi
}
\def\mathlig@checksuffix@#1#2#3\mathlig@end{%
\expandafter\ifx\csname mathlig@forw@#1#2\endcsname\relax\mathlig@dosuffix{#1}{#2}\fi
\mathlig@checksuffix{#1#2}#3\mathlig@end
}

\def\mathlig@dosuffix#1#2{%
\mathlig@appendcs{mathlig@toks@#1}{#2}%
\mathlig@dolig{#1}{#2}\mathlig@delim
}

\def\mathlig@dolig#1#2\mathlig@delim{%
 \mathlig@defcs{mathlig@look@#1#2}{%
 \mathlig@let@cs\mathlig@next{mathlig@forw@#1#2}\futurelet\mathlig@next@tok\mathlig@next}%
 \mathlig@defcs{mathlig@forw@#1#2}{%
  \mathlig@let@cs\mathlig@next{mathlig@back@#1#2}%
  \mathlig@let@cs\checker{mathlig@chck@#1#2}%
  \mathlig@let@cs\mathligtoks{mathlig@toks@#1#2}%
  \expandafter\ifx\expandafter\mathlig@delim\mathligtoks\mathlig@delim\relax\else
  \expandafter\checker\mathligtoks\mathlig@delim\fi
  \mathlig@next
 }%
 \mathlig@defcs{mathlig@toks@#1#2}{}%
 \mathlig@defcs{mathlig@chck@#1#2}##1##2\mathlig@delim{%
  \ifx\mathlig@next@tok##1%
   \mathlig@let@cs\mathlig@next@cmd{mathlig@look@#1#2##1}\let\mathlig@next\mathlig@gobble
  \fi
  \ifx\mathlig@delim##2\mathlig@delim\relax\else
   \csname mathlig@chck@#1#2\endcsname##2\mathlig@delim
  \fi
 }%
 \ifx\mathlig@delim#2\mathlig@delim\else
  \mathlig@defcs{mathlig@back@#1#2}{\csname mathlig@back@#1\endcsname #2}%
 \fi
}%

\catcode`@\mathlig@atcode

\mathchardef\ordinarycolon\mathcode`\:
\def\vcentcolon{\mathrel{\mathop\ordinarycolon}}
\mathlig{:=}{\vcentcolon=}
\mathlig{::=}{\vcentcolon\vcentcolon=}
\numberwithin{equation}{section}

\theoremstyle{plain}
\newtheorem{theo}{Theorem}[section]

\newtheorem{lem}[theo]{Lemma}
\newtheorem{prop}[theo]{Proposition}

\theoremstyle{definition}

\theoremstyle{rem}
\newtheorem*{ex*}{Example}
\theoremstyle{rem}
\newtheorem*{exs*}{Examples}
\theoremstyle{rem}
\newtheorem*{rems*}{Remarks}
\theoremstyle{rem}
\newtheorem*{rem*}{Remark}
\theoremstyle{rem}
\newtheorem{rem}{Remark}

\title[Rank one perturbations and Anderson-type Hamiltonians]{Rank one perturbations and Anderson-type Hamiltonians}

\author[C. Liaw]{Constanze Liaw,$^{1,*}$}

\address{$^{1}$Department of Mathematical Sciences, University of Delaware, 501 Ewing Hall, Newark, DE 19716, USA;
\newline
CASPER, Baylor University, One Bear Place \#97328,      
 Waco, TX  76798, USA.}
\email{\textcolor[rgb]{0.00,0.00,0.84}{liaw@udel.edu}}

\thanks{The work of C.~Liaw was supported by the National Science Foundation under the grant  DMS-1802682.}

\subjclass[2010]{Primary 47A55; Secondary 82B44, 81Q10.}

\keywords{rank one perturbations, Anderson-type Hamiltonian, Krein--Lifshits spectral shift, discrete random Schr\"odinger operator.}

\begin{document}

\maketitle

\begin{center}
\emph{To the memory of R.G.~Douglas. You were not only a vast source of knowledge. Words cannot fully express my appreciation for your steady advice, your unfaltering support and the many hours of mathematical discussions we shared.}
\end{center}

\begin{abstract}
Motivated by applications of the discrete random Schr\"odinger operator, mathematical physicists and analysts, began studying more general Anderson-type Hamiltonians; that is, the family of self-adjoint operators $$H_\omega = H + V_\omega$$ on a separable Hilbert space $\cH$, where the perturbation is given by $$V_\omega = \sum_n \omega_n (\fdot, \f_n)\f_n$$ with a sequence $\{\f_n\}\subset\cH$ and independent identically distributed random variables $\omega_n$.

We show that the the essential parts of Hamiltonians associated to any two realizations of the random variable are (almost surely) related by a rank one perturbation.
This result connects one of the least trackable perturbation problem (with almost surely non-compact perturbations) with one where the perturbation is `only' of rank one perturbations. The latter presents a basic application of model theory.

We also show that the intersection of the essential spectrum with open sets is almost surely either the empty set, or it has non-zero Lebesgue measure.
\end{abstract}

\section{\textbf{Introduction}}\label{s-rkfm}

In this spirit, let $H$ be a self-adjoint operator on a separable Hilbert space $\cH$. Let $\{\f_n\}\subset\cH$ be a sequence of linearly independent unit vectors in $\cH$, and let $\omega=(\omega_1, \omega_2, \hdots)$ consist of independent, identically distributed random variables $\omega_n$  corresponding to a probability measure on $\R$. Assume that the probability distribution satisfies Kolmogorov's 0-1 law (see Subsection \ref{ss-01} below).

Without going into details about the definition, the \emph{Anderson-type Hamiltonian} is an almost surely self-adjoint operator associated with
\begin{equation}\label{Model}
H_\omega = H + V_\omega \qquad\text{on }\cH, \qquad V_\omega = \sum\limits_n \omega_n (\fdot, \f_n)\f_n.
\end{equation}
In many applications the vectors $\f_n$ are mutually orthogonal. However, a priori, the definition allows the case of non-orthogonal vectors $\f_n$. And many of the properties that were originally proved for mutually orthogonal vectors immediately extend to this case.

Probably the most important special case of such Anderson-type Hamiltonians is the discrete Schr\"odinger operator with random potential on $l^2(\Z^d)$ given by
\begin{equation*}
Hf(x)=-\bigtriangleup f (x) = - \sum\limits_{|n|=1} (f(x+n)-f(x)), \quad \f_n(x)=\delta_n(x)=
\left\{\begin{array}{ll}1&x=n,\\ 0&\text{else,}\end{array}\right.
\end{equation*}
where each $\omega_n$ is distributed according to uniform distribution on the interval $[-c,c]$. That just means that each value in the interval occurs with equal probability.
Many Anderson models are special cases of an Anderson-type Hamiltonian.

From the perspective of classical perturbation theory \cite{katobook} the main difficulty is that the potential $V_\omega$ is almost surely a non-compact operator, implying that many results from classical perturbation theory cannot be applied here.

On the side we mention an important open problem concerning this perturbation family. The Anderson localization conjecture for weak disorder \cite{And1958, CFKS, KKL, L13, Banff} stands out as one that has been much attempted. The general question is whether or not an initially localized wave-packed will spread out over time, or remain localized in space as time moves on. Literature renders a variety of definitions what precisely \emph{localization} means. For example some definitions use the wave operator, while others formulate localization in terms of dynamical properties, or the persistence of a non-trivial absolutely continuous part (almost surely). The conjecture can be formulated with either of these definitions. For simplicity we choose the latter. To embed the conjecture we mention that for the discrete random Schr\"odinger operator in one dimension, $d=1$, operators $H_\omega$ are known to have trivial absolutely continuous parts (almost surely) whenever $c>0$. In higher dimensions, $d\ge 2$, there is a dimension dependent threshold $c_d$ above which the absolutely continuous parts vanish almost surely, and it is expected that for $d\ge3$ they prevail for small positive $c$. Now, it is conjectured that for $d=2$ the discrete random Schr\"odinger operator has vanishing absolutely continuous part (almost surely) whenever $c>0$; no matter how small.

In contrast to Anderson-type Hamiltonians stands the seemingly simple problem of perturbing a self-adjoint operator by an operator of rank one. Namely, for a self-adjoint operator $A$ on $\cH$ consider the \emph{family of self-adjoint rank one perturbations} by a vector $\f\in\cH$:
\[
A_\alpha = A+\alpha (\fdot, \f)\f,\qquad\alpha \in \R.
\]
(For details beyond this formal definition, see the discussion surrounding equation \eqref{e-rk1} below.) 

When the underlying Hilbert space $\cH$ is finite dimensional, we just need to keep track of the eigenvalues. However, for infinite dimensional $\cH$ intricate scenarios can occur that are closely connected with the boundary values of functions from model spaces.

In fact, the problem of rank one perturbations has connections to many interesting topics in analysis, such as model theory including deBranges--Rovnyak and Sz.-Nagy--Foia\c s model spaces \cite{cimaross, LT15, LTSurvey}, Nehari interpolation \cite{poltsara2006}, Carleson embeddings \cite{ross}, singular integral operators \cite{mypaper}, and truncated Toeplitz operators \cite{bessonov}.

With this in mind it becomes clear that, although rank one perturbations are the simplest from a perturbation theoretic perspective, their fine properties are extremely rich in nature. While Aronszajn--Donoghue theory captures much of the theory related to rank one perturbations, the picture is certainly not complete. For example, we do not know the singular continuous spectrum of the perturbed operator $A_\alpha$ in terms of properties of the unperturbed operator $A$, see e.g.~\cite{SIMREV}.

It was surprising when the Simon--Wolff criterion \cite{SIMWOL} on rank one perturbations was used to study localization properties of random Jacobi matrices \cite{SIMLOC}. These ideas were extended to Anderson-type Hamiltonians and refined \cite{ALP, JakLast2000, JakLast2006}. For example, it turns out that under mild conditions, \emph{any} non-zero vector is cyclic for the Anderson-type Hamiltonian almost surely.

In this manuscript, a new relationship between rank one perturbations and the essential parts of Anderson-type Hamiltonians is presented. In consideration of the great difference in the very nature of these two perturbation problems, this seems almost paradoxical. On the one hand this result restricts the spectral behavior of the Anderson-type Hamiltonians. On the other hand it shows the great complexity of the problem of rank one perturbations.

The proof at hand consists of constructing the spectral measures of the two operators. The Krein--Lifshits spectral shift function allows us to ensure that the hence constructed operators are indeed related by a rank one perturbations. These tools are based on similar observations made by A.~Poltoratski in \cite{alex2000}.

\subsection{Outline}
In Section \ref{s-prelims} we review related results from perturbation theory. We introduce and remind the reader of a few facts about the Krein--Lifshits spectral shift function for rank one perturbations, and we review on Kolmogorov's 0-1 law as well as its implications for Anderson-type Hamiltonians.

In Section \ref{s-firsttheo} we mention some simple known and some new results: In Subsection  \ref{subsection31}, we provide a short proof for two statements about the deterministic spectral structure of Anderson-type Hamiltonians.  In Subsection \ref{subsection3} we focus on the intersection of the essential spectrum with open sets and show that this intersection is almost surely either the empty set, or it has non-zero Lebesgue measure, see Theorem \ref{t-thirdtheo}.

In Section \ref{subsection3.2} we state and prove the main result (Theorem \ref{t-secondtheo}), which roughly says that the essential parts of $H_\omega$ and $H_\eta$ are almost surely with respect to the product measure $\p\times\p$ unitary equivalent modulo a rank one perturbation. % In the proof we use the Krein--Lifshits spectral shift function in the spirit of \cite{alex2000}.

%%%%%%%%%%%%%%%%%%%%%%%%%%%%%
\section{\textbf{Preliminaries}}\label{s-prelims}
%%%%%%%%%%%%%%%%%%%%%%%%%%%%%%%%%%%%%%%%%%%%%%%
%%%%%%%%%%%%%%%%%%%%%%%%%%%%%%%%%%%%%%%%%%%%%%%%%%%%%
\subsection{Perturbation Theory}
%%%%%%%%%%%%%%%%%%%%%%%%%%%%%%%%%%%%%%%%%%%%%%%%%%%%%
Perturbation theory is concerned with the general question:  Given some information about the spectrum of an operator $A$ what can be said about the spectrum of the operator $A+B$ for $B$ in some operator class? Depending on which class of operators the perturbation $B$ is taken from we obtain different results of spectral stability, i.e.~preservation of parts of the spectrum under such perturbations.

Since unitarily equivalent operators (i.e.~$UAU^{-1}=B$ for some unitary operator $U$) are of the same spectral type, we introduce the following notation.
We write $A\sim B$ for two operators $A$ and $B$ if the operators are unitary equivalent. The notation $$A\sim B (\Mod \text{Class }X)$$ is used if there exists a unitary operator $U$ such that $UAU^{-1}-B$ is an element of $\text{Class }X$. Here, $\text{Class }X$ can be any class of operators, e.g.~compact, trace class, or finite rank operators. 

For self-adjoint operators $A$ and $B$ let us recall the following well-known theorems that will be used in the proof of Theorem \ref{t-firsttheo} below.

\begin{theo}[Weyl--von Neumann, see e.g.~\cite{katobook}]\label{t-weylvn}
The essential spectra of two self-adjoint operators $A$ and $B$ satisfy $\sigma\ti{ess}(A)=\sigma\ti{ess}(B)$ if and only if $A\sim B (\Mod \text{compact operators})$.
\end{theo}

Here, the essential part of the spectrum is obtained by removing the isolated eigenvalues of finite multiplicity from the spectrum.

\begin{theo}[Kato--Rosenblum, see e.g.~ \cite{katobook}]\label{t-KR}
If for two self-adjoint operators we have $A\sim B (\Mod \text{trace class})$ then their absolutely continuous parts are equivalent, i.e.~$A\ti{ac}\sim B\ti{ac}$.
\end{theo}

We briefly explain how to recover the absolutely continuous part of an operator. First find a spectral measure $\mu$ (using the Spectral Theorem with respect to some minimal cyclic set of vectors) and take its Radon--Nikodym derivative $\frac{d\mu}{dx} = d\mu\ti{ac}$. The desired part of the operator is the one that corresponds to this absolutely continuous part of the measure. 

\begin{rem*}
For self-adjoint $A$ and $B$, Carey--Pincus \cite{CP} characterized when two operators are related by a rank one perturbation, that is, when we have $A\sim B~(\Mod \text{\em trace class})$. Of course, they must have unitarily equivalent absolutely continuous parts. Outside the continuous spectrum, they are only allowed discrete parts. And the discrete eigenvalues of $A$ and $B$ (counting multiplicity) must fall into three categories: (i) those eigenvalues of $A$ with distances from the joint continuous spectrum having finite $l^1$ norm (i.e.~are trace class), (ii) those eigenvalues of $B$ with distances from the joint continuous spectrum having finite $l^1$ norm, and (iii) eigenvalues of $A$ and $B$ that can be matched up (via a 1-1 and onto map) so that their differences have finite $l^1$ norm.
\end{rem*}

In the case of purely singular measures (i.e.~with trivial absolutely continuous part) the next theorem resembles a characterization for $A\sim B (\Mod \text{rank one})$.
Recall that two operators $A$ and $B$ are said to be completely non-equivalent, if there are no non-trivial closed invariant subspaces $\cH_1$ and $\cH_2$ of $\cH$ such that $A|\ci{\cH_1}\sim A|\ci{\cH_2}$. It is not hard to see that two operators are completely non-equivalent, if and only if their spectral measures are mutually singular. Here, we mean mutually singular in the sense of measure theory. That is, two measures $\mu$ and $\nu$ are said to be mutually singular, if there is a measurable set $B$ so that $\mu(B)=0$ and $\nu(\R\setminus B)=0$.

\begin{theo}[Poltoratski \cite{alex2000}]\label{t-Polt}
Let $K\subset\R$ be closed. By $I_1=(x_1;y_1), I_2=(x_2;y_2), \hdots$ denote disjoint open intervals such that $K=\R\backslash \bigcup I_n$. Let $A$ and $B$ be two cyclic self-adjoint completely non-equivalent operators with purely singular spectrum. Suppose $$\sigma(A)=\sigma(B)=K$$ and assume that for the pure point spectra (consisting of the eigenvalues) of $A$ and $B$ we have $$\sigma\ti{pp}(A)\cap\{x_1,y_1, x_2, y_2, \hdots\}=\sigma\ti{pp}(B)\cap\{x_1,y_1, x_2, y_2, \hdots\}=\varnothing.$$ Then we have $$A\sim B (\Mod \text{rank one}).$$
\end{theo}

The proof of our main result applies the latter theorem as well as Lemma \ref{l-dm} below which allows us to introduce absolutely continuous spectrum (while retaining precise control of the singular measures).

%%%%%%%%%%%%%%%%%%%%%%%%%%%%%%%%%%%%%%%%%%%%%%%%%%%%%
\subsection{Cauchy transform and rank one perturbations}
%%%%%%%%%%%%%%%%%%%%%%%%%%%%%%%%%%%%%%%%%%%%%%%%%%%%%
The deep connection between operator theory and the Cauchy transform
\begin{align*}
K\tau(z)=\frac{1}{\pi}\int_\R \frac{d\tau(t)}{t-z}\,,
\qquad
z\in \C_+,
\end{align*}
of an operator's spectral measure $\tau$ is well studied. This relationship is frequently used to learn about the spectral properties of the operator under investigation.
The connection between operator theory and the Cauchy transform and the spectral theory of rank one perturbations is particularly well developed \cite{cimaross, LT15, LTSurvey, mypaper, poltsara2006}. This connection is one of our major ingredients. Here we merely recall the results that are applied later in this article.

It is well-known that the density/weight function $w\in L^1$ of the absolutely continuous part of the measure can be recovered via
\begin{align}\label{e-acIm}
&d\tau\ti{ac}(x)= w dx =  \lim_{y\downarrow0}\Im\, K\tau (x+iy) \,dx,
\quad
x\in \R,
\end{align}
where $\Im$ denotes the imaginary part.

In Aleksandrov--Clark Theory, the following result plays an essential role.

\begin{theo}\label{t-polt}(Poltoratski \cite{NONTAN}, also see \cite{NPPOL}). Let $\tau$ and $\tilde\tau$ be two non-negative measures on the real line such that $\tilde \tau = f \tau +\tilde \tau\ti{s}$. Then
\[
\frac{K\tilde\tau}{K\tau}(x+i\e) \stackrel{\e\to0}{\longrightarrow} f(x)
\qquad \tau\ti{s}-\text{almost everywhere}.
\]
\end{theo}

Here we always work with measures that satisfy Poisson integrability $\int\frac{d\tau(t)}{t^2+1}<\infty$. Especially when dealing with rank one perturbations, we do often encounter measures with $\int\frac{d\tau(t)}{|t|+1}=\infty$. In order to avoid difficulties with convergence it is standard to introduce an alternative definition of the Cauchy transform
\begin{align*}
K_1\tau(z)=\frac{1}{\pi}\int_\R \left(\frac{1}{t-z}-\frac{t}{t^2+1}\right)\,d\tau(t),
\qquad
z\in \C_+.
\end{align*}
We use both $K\tau$ and $K_1\tau$ below.
Notice that the two behave alike locally, as the integrand $-\frac{t}{t^2+1}$ is uniformly bounded on $\R$. 
Although it will not play a role later on it is worth mentioning that (for $\tau$ such that $K\tau$ is defined on $\C_+$) the real part of $K_1\tau$ differs from the conjugate Poisson integral by a finite additive constant. 

The advantage of introducing this alternative definition is that it is possible to define $K_1\tau$ for more general measures $\tau$.  Indeed, since $\frac{1}{t-z}-\frac{t}{t^2+1}$ behaves like $t^{-2}$ as $t\to\infty$, we can work with Poisson integrable measures $\tau$ and do not need to assume the stronger condition $\int\frac{d\tau(t)}{|t|+1}<\infty$.

Let $A$ be a self-adjoint (possibly unbounded operator) on a Hilbert space $\cH$. Let $\f$ be such that the corresponding rank one perturbation will be form bounded, i.e.~$\|(1+|A|)^{-1/2}\f\|\ci{\cH}<\infty$; see \cite{mypaper} and its references for more information. Then we can use quadratic forms to define the family of rank one perturbations via the formal expression
\begin{align}\label{e-rk1}
A_\alpha&=A+\alpha (\cdot,\f)\f,
\qquad \alpha\in \R.
\end{align}
Only focussing on the interesting part of the perturbation problem, we assume that $\f$ is a cyclic vector for $A$, i.e.$$
\cH = \overline{\spa\{(A-z\OID)^{-1} \f: z\in \C\backslash \R}\}.$$ To see that we are not restricting generality, notice that on the orthogonal complement of the invariant subspace $\overline{\spa\{(A-z\OID)^{-1} \f: z\in \C\backslash \R}\}$ for $A$ and $A_\alpha$ in $\cH$, operator $A_\alpha$ is independent of $\alpha$.

In our setting, it is well-known that $\f$ is also a cyclic vector of the operator $A_\alpha$ for all $\alpha\in \R$. By $\mu_\alpha$ denote the spectral measure of $A_\alpha$ with respect to $\f$. In other words, invoking the Spectral Theorem, $\mu_\alpha$ is given by $$((A_\alpha-z\OID)^{-1}\f,\f)\ci{\cH}=\int_\R \frac{d\mu_\alpha(t)}{t-z}\qquad\text{for all }z\in \C\backslash\R.$$
We use the notation $\mu=\mu_0$.

With the resolvent formula, it is not difficult to see that the Cauchy transforms of the measures $\mu$ and $\mu_\alpha$ of the rank one perturbation \eqref{e-rk1} are related via the Aronszajn--Krein formula
\begin{align}\label{e-AK}
K\mu_\alpha&=\frac{K\mu}{1+\pi\alpha K\mu}\,,
\end{align}
also see \cite[Equation (11.13)]{SIMREV}.

Aronszajn--Donoghue theory (see e.g.~\cite[Section 12.2]{SIMREV}) provides a good picture of the spectrum of the perturbed operator for rank one perturbations. One of its  intriguing results says that the singular part of rank one perturbations must move when we change the perturbation parameter $\alpha$:

\begin{theo}[Aronszajn--Donoghue]\label{t-AD}
For coupling constants $\alpha\neq\beta\in \R$, the singular parts of the corresponding spectral measures $\mu_\alpha$ and $\mu_\beta$ are mutually singular, i.e.~$(\mu_\alpha)\ti{s}\perp(\mu_\beta)\ti{s}$.
\end{theo}

This result was proved by Aronszajn for Sturm--Liouville operators with varying boundary conditions \cite{Aronszajn} and by Donoghue in the abstract setting of rank one perturbations \cite{Donoghue}.

Another result within this theory gives a necessary condition for a point to be in the essential support of the singular spectrum of $A_\alpha$. The theorem in this form can easily be extracted from Theorem 6 of \cite{Donoghue}, which states that the set $\{x:\lim_{y\downarrow0}K\mu(x+iy)= -\alpha^{-1}\}$ is a carrier for $(\mu_\alpha)\ti{s}$ (meaning that $(\mu_\alpha)\ti{s}$ is trivial outside that set).

\begin{theo}\label{t-AS}
We have $(\mu_\alpha)\ti{s}(\{x:\lim_{y\downarrow0}K\mu(x+iy)\neq -\alpha^{-1}\})=0$.
\end{theo}

\subsection{Essential support of the absolutely continuous part of a  measure}
In order to define one of the objects of interest, we isolate the limit supremum from the symmetric definition of the Radon-Nikodym derivative.

In this spirit, we let $\tau$ be a Borel measure on $\R$. Fix $\e>0$ and consider the Borel function $x\mapsto D_\e\tau(x)$ where
\[
D_\e\tau(x):=
\frac{\tau([x-\e, x+\e])}{2\e}.
\]
Note that the denominator equals the Lebesgue measure of  interval $[x-\e, x+\e]$.

The essential support of the absolutely continuous part of a Borel measure $\tau$ (on $\R$) is  given by
\begin{align}\label{d-esupp}
\esupp\tau\ti{ac}=\left\{x\in\R:0<\limsup\limits_{\e\to0} D_\e\tau(x)<\infty\right\}.
\end{align}

\begin{rem}\label{ess-suppA}
In order to embed this into classical theory, we mention that the Radon-Nikodym derivative of $\tau$ exists at $x$, if and only if $$\limsup\limits_{\e\to0} D_\e\tau(x) = \liminf\limits_{\e\to0} D_\e\tau(x)<\infty.$$
\end{rem}

\begin{rem}\label{ess-supp}
Since the Radon--Nikodym derivative exists almost everywhere (with respect to Lebesgue measure) two operators satisfy $A\ti{ac}\sim B\ti{ac}$ if and only if the essential supports of the absolutely continuous parts of their spectral measures are equal up to a set of measure zero.
Indeed, as described in \cite[Section 12.1]{SIMREV}, two absolutely continuous measures $f(x)dx$ and $g(x)dx$ are equivalent if and only of the symmetric difference of the sets $\{x\mid f(x)\neq 0\}$ and $\{x\mid g(x)\neq 0\}$ has Lebesgue measure zero. And the operators that act as multiplication by the independent variable $M_x$ on $L^2(f(x)dx)$ and $L^2(g(x)dx)$ are unitarily equivalent if and only if the measures $f(x)dx$ and $g(x)dx$ are equivalent. It remains to apply Remark \ref{ess-suppA}.
\end{rem}

\begin{rem}
The same arguments as in Remark \ref{ess-supp} also imply that the essential support of the absolutely continuous part of an operator's spectral measure is up to a set of measure zero independent of the choice of cyclic vector (used in the spectral theorem).
\end{rem}

It is worth presenting a simple example to demonstrate that $\esupp\tau\ti{ac}\subsetneq\supp\tau\ti{ac}$ may happen:
\begin{ex*}
Let $\tau$ be the measure given by the sum of Lebesgue measures on intervals that have all rational points of $[0,3]$ as centers and with width $2^{-n+1}$. Namely, with an enumeration $\{q_n\}$ of these rational points, let
\[
d\tau(x)=
\sum_{n\in \N}
\chi\ci{[q_n-2^{-n}, q_n+2^{-n}]}(x)dx.
\]
The sum of the interval width $\sum_{n\in\N} 2^{-n+1} = 2$, so that the Lebesgue measure of the essential support satisfies the crude estimate $|\esupp\ti{ac}\tau|\le2$. On the other hand, the rationals are dense in $[0,3]$ and so $3\le|\supp\tau\ti{ac}|$. In fact, as $0$ and $3$ are centers of some intervals, we have $3<|\supp\tau\ti{ac}|$. In any case, we have $\esupp\tau\ti{ac}\subsetneq\supp\tau\ti{ac}$.
\end{ex*}

\subsection{Krein--Lifshits Spectral Shift for Rank One Perturbations}\label{ss-KL}
In this section, we briefly present the Krein--Lifshits spectral shift function and its properties for rank one perturbations. More detailed explanations, examples and proofs can be found in \cite{alex1998} and the references therein.

Consider the rank one perturbations $A_\alpha$ given by \eqref{e-rk1} and their spectral measures $\mu_\alpha$ corresponding to the cyclic vector $\f$.

Since the spectral measure $\mu$ is non-negative, the Cauchy transform $K\mu(z)$ is Herglotz, i.e.~its imaginary part is non-negative for $z\in\C_+$. 
For every $\alpha \in \R$ it is hence possible to find an essentially bounded by $-\pi<u(t)\le \pi$, $t\in \R$, function and a constant $c\in \R$ such that
\begin{align}\label{e-nuu}
&1+ \pi \alpha K\mu = e^{K_1u+c}\,.
\end{align}
see e.g.~\cite[Section VIII.1]{MP}.
To better understand this formula, recall that the angular boundary values of the Cauchy transform exist almost everywhere with respect to the Lebesgue measure. Now think of $K_1 u$ as the analytic upper half-plane extension of $u$. So that for $\alpha>0$ (we can always re-label $A$ and $A_\alpha$ so that $\alpha>0$), function $u$ can equivalently be defined via the principal argument
\begin{align}\label{e-arc}
u=\text{arg} (1+\pi\alpha K\mu).
\end{align}
Function $u$ is called the Krein--Lifshits spectral shift of the rank one perturbation $A_\alpha$.
Since $K\mu$ is Herglotz, the range of $u$ is contained in $[0,\pi].$
Indeed, consider the logarithm of \eqref{e-nuu}, take its imaginary part and recall the relation \eqref{e-acIm}.

By breaking $K\mu$ in \eqref{e-arc} into real and imaginary part $K\mu =iP\mu-Q\mu$ (where $P$ denotes the Poisson integral and $Q$ denotes the conjugate Poisson integral), it becomes clear that the singularity of the integrand causes $u$ to jump from $0$ to $\pi$ at isolated points of $\supp\mu\ti{s}$.

In the non-isolated case, a characterization of the point masses of $\mu$ and $\mu_\alpha$ is included in \cite[Section VIII.5]{MP}.

Using the Aronszajn--Krein formula \eqref{e-AK} we obtain a relation between the shift function and the measure $\mu_\alpha$:
\begin{align*}%\label{e-nun}
&1- \pi \alpha K\mu_\alpha= e^{-K_1u-c}\,.
\end{align*}
And the analog
\begin{align}\label{e-arg}
u=-\text{arg} (1-\pi\alpha K\mu_\alpha)
\end{align}
of \eqref{e-arc} for $\mu_\alpha$ implies that $u$ drops from $\pi$ to $0$ at isolated points of $\supp(\mu_\alpha)\ti{s}$.

So in essence, each family of spectral measures $\{\mu_\alpha\}\ci{\alpha\in \R}$ corresponds to some Krein--Lifshits spectral shift function $u$.

Further the set where $u\in (0, \pi)$ and not equal to one of the endpoints of the inverval is equal (up to a set of Lebesgue measure zero) to $\esupp(\mu)\ti{ac}$. In particular, it follows that
\[
\esupp(\mu)\ti{ac}=\esupp(\mu_\alpha)\ti{ac}.
\]

\begin{rem}
These observations about the relationship between the spectrum of $A$ and $A_\alpha$, and the behavior of $u$ give an alternative proof for the fact that the discrete spectrum of two purely singular operators in the same family of rank one perturbations must be interlacing. In absence of absolutely continuous spectrum, $u$ can only take on the values $0$ and $\pi$, so that the Krein--Lifshits spectral shift essentially jumps from $0$ to $\pi$ and then back.
\end{rem}

Vice versa, it is well-known that for fixed $\alpha>0$ any measurable function $u$ which is essentially bounded by $0\le u\le \pi$ is the Krein--Lifshits spectral shift of the rank one perturbation $M_\mu+\alpha (\cdot,\ID)\ID$ of the multiplication operator $M_\mu$ by the independent variable on $L^2(\mu)$. In fact, given such a function $u$ and $\alpha>0$ we obtain a unique pair of measures $\mu$ and $\nu=\mu_\alpha$ if we impose a normalization condition on the measures. For $\alpha =1$ we say that the measures $\mu$ and $\nu$ correspond to $u$.

%%%%%%%%%%%%%%%%%%%%%%%%%%%%%%%%%%%%%%%%%%%%%%%%%%%%%%
\subsection{Kolmogorov's 0-1 law and Anderson-type Hamiltonians}\label{ss-01}
%%%%%%%%%%%%%%%%%%%%%%%%%%%%%%%%%%%%%%%%%%%%%%%%
Consider triples $(\Omega, \mathcal{A}, \p)$ of probability spaces, where $\Omega = \R^\infty$ consists of countably many copies of $\R$ and where $\p$ is a countable product of equal probability measures. We let $\omega=(\omega_1, \omega_2, \hdots)\in\Omega$ be taken in accordance with $\p$. 

Here we consider only those probability measures $\p$ that satisfy Kolmogorov's 0-1 law. Namely, properties that are invariant under changing finitely many of the $\omega_n$ are enjoyed with probability 0 or 1.
This is particularly useful here, because perturbation theory tells us that many properties are independent under finite rank perturbations.

Specifically, we use:
\begin{prop}[Kolmogorov's 0-1 law applied to Anderson-type Hamiltonians]\label{obs}
Consider the Anderson-type Hamiltonian $H_\omega$ given by \eqref{Model}. Assume that the probability distribution $\p$ satisfies the 0-1 law. Then those spectral properties that are invariant under finite rank perturbations are enjoyed by $H_\omega$ almost surely or almost never.
\end{prop}

%%%%%%%%%%%%%%%%%%%%%%%%%%%%%%%%%%%%%%%%%%%%%%%%%%%%%%%%%%%
\section{\textbf{Deterministic spectral structure}}\label{s-firsttheo}
%%%%%%%%%%%%%%%%%%%%%%%%%%%%%%%%%%%%%%%%%%%%%%%%%%%%%%%%%%%
%%%%%%%%%%%%%%%%%%%%%%%%%%%%%%%%%%%%%%%%%%%%%%%%%
\subsection{Deterministic absolutely continuous part and essential spectrum}\label{subsection31}
%%%%%%%%%%%%%%%%%%%%%%%%%%%%%%%%%%%%%%%%%%%%%%%%%%
\begin{theo}\label{t-firsttheo}
Let $H_\omega$ be given by \eqref{Model}. Assume the hypotheses of Section \ref{s-rkfm} and assume that $\p$ satisfies the Kolmogorov 0-1 law. Then almost surely with respect to the product measure $\p\times\p:$
\begin{itemize}
\item[1)] $(H_\omega)\ti{ac}\sim(H_\eta)\ti{ac}$ and
\item[2)] $H_\omega\sim H_\eta(\Mod \text{compact operator})$.
\end{itemize}
\end{theo}

While the statement in item 1) is known (see \cite[Corollary 1.3]{JakLast2000}), we present a short proof for the convenience of the reader and since the proof structure also underlies the proof of the statement in item 2).

\begin{proof}
The words `almost surely' (`almost never') in this proof refer to almost surely (almost never) with respect to the product measure $\p\times\p$.

Let $H_{\widetilde\omega}$ denote finite rank perturbations of $H$, i.e.~$\widetilde \omega = (\widetilde \omega_1, \widetilde \omega_2, \hdots)$ with $\widetilde \omega_n\neq 0$ only for finitely many $n$.
In particular, $ H_{\widetilde\omega}$ are compact and trace class perturbations of $H$.

To show the statement in item 1). Without loss of generality, let $\mu_\omega$ denote the `fiber' of the spectral measure of $H_\omega$ for which $\esupp\mu_\omega$ is maximal with respect to the inclusion of sets. (Alternatively, one can think of $\mu_\omega$ as the associated scalar-valued spectral measure. This can also be obtained by taking the trace of a matrix-valued spectral measure.) Let $\mu_{\widetilde\omega}$ be the analog measure for $ H_{\widetilde\omega}$.

By the Kato--Rosenblum theorem (see Theorem \ref{t-KR}) and Remarks \ref{ess-suppA} and \ref{ess-supp}, for almost every $x\in \R$ we have $x\in\esupp (\mu_{(0,0,0,\hdots)})\ti{ac}$ if and only if $x\in\esupp (\mu_{\widetilde\omega})\ti{ac}$. By virtue of the Kolmogorov 0-1 law (see Proposition \ref{obs}),  for almost every $x\in \R$ we have $x\in \esupp(\mu_\omega)\ti{ac}$ almost surely or almost never. The set (up to a set of measure zero) of points $x$ for which the latter is almost surely true is hence deterministic and the statement in item 1) is proven.

Item 2) follows in analogy via the Weyl--von Neumann theorem (see Theorem \ref{t-weylvn}) replacing Theorem \ref{t-KR}.
\end{proof}

\begin{rem}
(a) In fact, we have proved the stronger -- than item 1) of Theorem \ref{t-firsttheo} -- statement that the essential support of the absolutely continuous spectrum is a deterministic set (up to a set of Lebesgue measure zero). Namely, for some measurable set $A\subset \R$ we have that the symmetric difference $$A\bigtriangleup \esupp(\mu_\omega)\ti{ac}$$ has Lebesgue measure zero $\p$ almost surely $\omega$.\\
(b) Similarly for item 2) of Theorem \ref{t-firsttheo}, it follows that there exists a deterministic set $K$ such that $K=\sigma\ti{ess}(H_\omega)$ almost surely.\\
(c) Although the perturbation $V_\omega$ is almost surely (with respect to $\p$) a non-compact perturbation, there is still a deterministic set $K = \sigma\ti{ess}(H_\omega)$ for $\p$ almost all $\omega$.
\end{rem}

%%%%%%%%%%%%%%%%%%%%%%%%%%%%%%%%%%%%%%%%%%%%%%%%%
\subsection{Intersection of the essential spectrum with open sets}\label{subsection3}
%%%%%%%%%%%%%%%%%%%%%%%%%%%%%%%%%%%%%%%%%%%%%%%%%%
Assume the setting of Theorem \ref{t-firsttheo}.
Recall that $\sigma\ti{ess}(H_\omega)$ is a deterministic set, by item 2) of Theorem \ref{t-firsttheo}.

\begin{theo}\label{t-thirdtheo}
Assume the hypotheses of Theorem \ref{t-firsttheo} and assume that $\p$ is a product of absolutely continuous measures. Let $O$ be an open set and let $X=O\cap \sigma\ti{ess}(H_\omega)$. Then almost surely $$\text{either }X =\varnothing,\text{ or the Lebesgue measure }|X|>0.$$
\end{theo}

\begin{proof}
Assume $|X|=0$ and $X\neq \varnothing$. Take $x\in X$.

Since $O$ is open, there exists $\e>0$ such that the interval $(x-\e,x+\e)\subset O$. Consider $X_\e = X\cap (x-\e,x+\e)$. Clearly we have $|X_\e|=0$.

Recall item 1) of Theorem \ref{t-firsttheo}. This implies that almost surely
\[
(\mu_\omega)\ti{ac}((x-\e, x+\e))=(\mu_\omega)\ti{ac}(X_\e)=0.
\]

In virtue of Lemma \ref{l-NoSAS} below $(\mu_\omega)\ti{s}(X_\e)=0$ almost surely.

Therefore $x\notin \sigma\ti{ess}(H_\omega)$ almost surely, in contradiction to the fact that $x\in X$. Hence almost surely either $X=\varnothing$ or $|X|>0$.
\end{proof}

\begin{lem}\label{l-NoSAS}
Assume the hypotheses of Theorem \ref{t-firsttheo} and assume that $\p$ is a product of absolutely continuous measures $\mu_k$. If set $A\subset \R$ satisfies $|A|=0$, then we have $(\mu_\omega)\ti{s}(A)=0$ almost surely.
\end{lem}

\begin{proof}
Recall that $\p$ is a product of absolutely continuous measures $\mu_k$.

Assume that $(\mu_\omega)\ti{s}(A)>0$ with positive probability. Then (for arbitrary $k\in \N$) there exist $\omega_0$ and $\mathcal{X}\subset \R$ such that $\mu_k(\mathcal X)>0$ and such that for all $\alpha\in \mathcal X$ we have $(\mu_{\omega_\alpha})\ti{s}(A)>0$ where $\omega_\alpha = \omega_0+\alpha \delta_k$.

But this contradicts the Aronszajn--Donoghue Theorem \ref{t-AD} for rank one perturbations. Notice that $\mathcal X$ contains at least two points, since all $\mu_k$ are absolutely continuous.
\end{proof}

%%%%%%%%%%%%%%%%%%%%%%%%%%%%%%%%%%%%%%%%%%%%
\section{\textbf{Almost sure unitary equivalence modulo a rank one perturbation}}\label{subsection3.2}
%%%%%%%%%%%%%%%%%%%%%%%%%%%%%%%%%%%%%%%%%%%%%%%
The main result of this paper, see Theorem \ref{t-secondtheo} below, says that the essential parts of two Anderson-type Hamiltonians are unitarily equivalent modulo a rank one perturbation. Its proof relies on constructing an appropriate Krein--Lifshits spectral shift function.

By $\partial S$ we denote the boundary of a given  set $S$, and by $|\fdot|$ denote the Lebesgue measure.

\begin{theo}\label{t-secondtheo}
Assume the hypotheses of Theorem \ref{t-firsttheo}. Assume that $(H_\omega)\ti{ess}$ is cyclic almost surely (with respect to  $\p$) and $\p=\Pi_k \mu_k$ is a product measure of purely absolutely continuous measures $\mu_\omega$ on $\R$. Let $\mu$ denote the spectral measure of the operator $(H_\omega)\ti{ess}$ with respect to some cyclic vector. If $|\partial \esupp (\mu_\omega)\ti{ac} |=0$ almost surely, then $$(H_\omega)\ti{ess}\sim (H_\eta)\ti{ess}(\Mod \text{rank one})$$ almost surely with respect to the product measure $\p\times\p$.
\end{theo}

On the one hand, this result greatly restricts the possible deterministic properties of Anderson-type Hamiltonians. On the other hand, it tells us how `wild' rank one perturbations can be.

Recall that the essential spectrum comes about from removing from the spectrum all isolated point masses that have finite multipilicity. Further recall that the absolutely continuous and singular parts of the spectrum arise from Lebesgue decomposition of its spectral measure, $\mu = \mu\ti{ac}+\mu\ti{s}$. A particular decomposition of the operator is then obtained through unitary equivalence with the particular decomposition of the spectral representation. (That is, on the spectral representation side, the $L^2(\mu)$ space is orthogonally decomposed in accordance with the particular spectral decomposition, the multiplication operator is restricted to these invariant subspaces, and the decomposition of the operator is carried over via unitary equivalence.)

\begin{rem}
(a) If a family of Anderson-type Hamiltonians possesses a weak Anderson localization property (namely, if there is no absolutely continuous spectrum almost surely), then the hypotheses of cyclicity and $|\partial \esupp (\mu_\omega)\ti{ac} |=0$ hold automatically. Indeed, the restricted operator $(H_\omega)\ti{s}$ is cyclic almost surely by Theorem 1.2 of \cite{JakLast2006}, and also recall that the operators $(H_\omega)\ti{ac}$ and $(H_\omega)\ti{s}$ are completely non-equivalent because the essential supports of their spectral measures are mutually singular. Similarly, almost sure cyclicity of $(H_\omega)\ti{ac}$ implies the almost sure cyclicity of $(H_\omega)\ti{ess}$.\\
(b) In the conclusion of this result it is necessary to restrict to the essential parts of the operators. The statement $H_\omega\sim H_\eta(\Mod \text{rank one})$ is not true, since the finite isolated point spectra of $H_\omega$ and $H_\eta$ might not interlace. This intertwining is one of the necessary conditions for two operators to be unitarily equivalent up to rank one perturbation. In fact, between two points in the discrete spectrum of $H_\omega$ there may be any number of points from the discrete spectrum of $H_\eta$ (almost surely).\\
(c) Theorem \ref{t-secondtheo} cannot be concluded trivially by using Theorem \ref{t-Polt}, plus item 1) of Theorem \ref{t-firsttheo} and then separating the singular from the absolutely continuous part. This can be seen by counterexample: Embedded singular spectrum can occur for one operator, but not for the other (with positive probability). In particular, the absolutely continuous spectrum of $(H_\omega)\ti{ess}$ may have dense embedded singular spectrum, and $(H_\eta)\ti{ess}$ has purely absolutely continuous spectrum. In this case, the singular parts of $(H_\omega)\ti{ess}$ and $(H_\eta)\ti{ess}$ are not unitarily equivalent up to rank one perturbations (as they would have to interlace).\\
(d) We expect that relaxing the hypotheses of the theorem from $(H_\omega)\ti{ess}$ is cyclic to assuming that it has finite multiplicity $m$ would yield the conclusion $(H_\omega)\ti{ess}\sim (H_\eta)\ti{ess}(\Mod \text{rank }m)$.
\end{rem}

The proof of Theorem \ref{t-secondtheo} uses Poltoratski's result on a characterization of rank one perturbations in terms of the spectrum (Theorem \ref{t-Polt}) as well as the following lemma which will allow us to introduce absolutely continuous spectrum while retaining precise control of the singular measures.

\begin{lem}\label{l-dm}
Let $u$ be a Krein--Lifshits spectral shift function with range in the set $\{0,\pi\}$. Let $\mu$ and $\nu$ be the corresponding spectral measures. Take an open set $O\subset \R$ such that $|O|<\infty$. For $c>0$ define a new shift function by
\[
\tilde u (x)=  
\left\{
\begin{array}{ll}
 u(x) &\text{on }\R\backslash O\\
 \left|u(x)-\min\{\dist(\R\backslash O, x),\pi/2\}\right|,& \text{if } x\in O.
\end{array}
\right.
\]

For the measures $\tilde \mu$ and $\tilde \nu$ that correspond to $\tilde u$, we have the equivalence of measures $\tilde \mu|\ci{\R\backslash O} \sim \mu|\ci{\R\backslash O}$ and $\tilde\nu|\ci{\R\backslash O}  \sim \nu|\ci{\R\backslash O}$.
\end{lem}

\begin{proof}
For $t\in \R\backslash O$ we have
\[
 |K_1(u-\tilde u)(t)|
 \le
 \int_O
 \left|
 \frac{u(x)-\tilde u(x)}{t-x}
 \right|
 dx
\le
\int_O \frac{\dist(\R\backslash O, x)}{|t-x|} dx
\le
|O|,
\]
and with \eqref{e-nuu}, it follows that
\[
0<c<\frac{1+\pi  K\tilde\mu}{1+\pi  K\mu} < C<\infty \qquad
 \mu|\ci{\R\backslash O}-\text{almost everywhere}.
\]
(Since $\tilde \mu$ and $\tilde \nu$ correspond to $\tilde u$, we have by convention $\alpha = 1$.)

By definition $\mu|\ci{\R\backslash O}$ and $\tilde\mu|\ci{\R\backslash O}$ are purely singular. Therefore, we have
\begin{align}\label{e-cC}
 0<\tilde c<\frac{K\tilde\mu}{K\mu} < \tilde C<\infty \qquad
 \mu|\ci{\R\backslash O}-\text{almost everywhere}.
\end{align}

If (on $\R\backslash O$) measure $\mu$ has a part that is singular with respect to $\tilde\mu$ (denote it by $\eta$), then the ratio of Cauchy integrals $ \frac{K \tilde\mu}{K\mu}$ tends to zero with respect to $\eta$ almost everywhere. This contradicts the lower bound of the last estimate \eqref{e-cC}.
Hence $\mu|\ci{\R\backslash O}$ must be absolutely continuous with respect to $\tilde\mu|\ci{\R\backslash O}$.

The other direction -- that $\tilde\mu|\ci{\R\backslash O}$ is absolutely continuous with respect to $\mu|\ci{\R\backslash O}$ -- follows in analogy and we have proven $$\tilde\mu|\ci{\R\backslash O} \sim \mu|\ci{\R\backslash O}.$$

The result for $\nu$ can be proven in analogy. 
\end{proof}

\begin{proof}[Proof of Theorem \ref{t-secondtheo}.] Most of this proof is to be understood almost surely with respect to the product measure $\p\times\p$, although this might not be stated everywhere explicitly.

By $\mu$ denote the spectral measure of the operator $(H_\omega)\ti{ess}$ with respect to some cyclic vector and similarly for $\nu$ and $(H_\eta)\ti{ess}$, where $(\omega,\eta)$ is distributed according to $\p\times\p$. It is worth mentioning that the spectral measures of an operator corresponding to any two cyclic vectors are equivalent.

In virtue of Lemma \ref{l-MSAS} (below) we have that $\mu\ti s\perp \nu\ti s$ almost surely with respect to product measure.
%Consider the measure $\tau$ on $\R$ given by $d\tau(t)=(t^2+1)^{-1}dt$.

The goal is to produce a spectral shift function with corresponding spectral measures that are equivalent to the spectral measures $\mu$ and $\nu$, respectively. This is done by construction of auxiliary measures $\mu_1$ and $\nu_1$ that behave like $\mu$ and $\nu$ on the singular parts. And in a second step we modify these auxiliary measures to obtain the desired absolutely continuous parts. In the end, we verify that we did not destroy the good singular behavior that the auxiliary measures had.

By item 1) of Theorem \ref{t-firsttheo}, the symmetric difference $$\esupp\mu\ti{ac}\bigtriangleup\esupp\nu\ti{ac}$$ is a set of measure zero (almost surely with respect to the product measure). Let us denote the intersection of these sets by $F=\esupp\mu\ti{ac}\cap\,\esupp\nu\ti{ac}$.
Notice that by the hypothesis, without loss of generality, we can assume $|\partial \esupp\mu\ti{ac}|=|\partial \esupp\nu\ti{ac}|=0$. A simple set theoretic argument shows that $|\partial F|=0$.

Further, by item 2) of Theorem \ref{t-firsttheo} and the Weyl--von Neumann theorem, Theorem \ref{t-weylvn}, their essential spectra satisfy $\sigma\ti{ess}(H_\omega)=\supp\mu=\supp\nu$. Let us denote this set by
\[
E = \sigma\ti{ess}(H_\omega).
\]

First observe that, by definition of $E$, operators $(H_\omega)\ti{ess}$ and $(H_\eta)\ti{ess}$ have dense purely singular spectrum on the set $E\backslash \clos (F)$.
By the definition of $F$ and since $|\partial F|=0$, it is possible to choose two purely singular measures $\mu'$ and $\nu'$ such that:
\begin{itemize}
\item $\mu'$ and $\nu'$ are mutually singular ($\mu'\perp\nu'$),
\item $\mu'|\ci{\R\backslash (F\backslash \partial F)}=\nu'|\ci{\R\backslash (F\backslash \partial F)}=0$, and so that
\item $\mu_1=\mu\ti{s}+\mu'$ and $\nu_1=\nu\ti{s}+\nu'$ have dense (alternating) spectrum on $E$.
\end{itemize}

The rough idea is that $\mu_1|\ci{\R\backslash (F\backslash \partial F)}$ and $\nu_1|\ci{\R\backslash (F\backslash \partial F)}$ are essentially what we are looking for. Further, $\mu_1$ and $\nu_1$ are spectral measures of operators that are rank one perturbations of one another. We still need to modify these measures on $F\backslash \partial F$, in order to ensure that the constructed measures are equivalent to $\mu$ and $\nu$ also on $F$.

By Theorem \ref{t-Polt}, the measures $\mu_1$ and $\nu_1$ possess a spectral shift function $u_1$, i.e.~there exists a function $u_1$ which is essentially bounded by $0\le u_1 \le\pi$ and such that
\[
u_1=\arg(1+\pi K\mu_1)=-\arg(1-\pi K\nu_1).
\]
Note that the hypothesis that there are no point masses at the endpoints is satisfied almost surely. So we can assume this condition without loss of generality.

In order to destroy the artificially created singular spectrum and introduce the appropriate absolutely continuous spectrum, we define
\begin{equation*}
u_2(x)
=
\left\{
\begin{array}{ll}
u_1(x), & \text{if }x \in \R\backslash (F\backslash \partial F),\\
|u_1(x) - \min \{\dist(\R\backslash (F\backslash \partial F) , x), \pi/2 \}|, & \text{if }x\in F\backslash \partial F,
\end{array}
\right.
\end{equation*}
and let $\mu_2$ and $\nu_2$ be the measures corresponding to $u_2$.

It remains to prove that $\mu_2\sim \mu$ and $\nu_2\sim \nu$. We will explain the equivalence of $\mu_2$ and $\mu$. The same fact for $\nu$ follows in analogy.

Let us begin with the absolutely continuous parts.
Recall  that $|\partial F|=0$. So on the set $F$ we have $u_2\in(0,\pi)$ Lebesgue almost everywhere. By equations \eqref{e-arc}, \eqref{e-arg} and \eqref{e-acIm}, it follows that $\frac{d\mu_2}{dx}(x)>0$ and $<\infty$ for Lebesgue almost all $x\in F$. This means that $$(\mu_2)\ti{ac}|\ci F\sim (\mu)\ti{ac}|\ci F.$$

The equivalence of the absolutely continuous part on $\R\backslash F$ follows similarly from the fact that $u_2$ takes only the values $0$ or $\pi$ on $\R\backslash F$.

We have shown that $(\mu_2)\ti{ac}\sim \mu\ti{ac}$. And by the same reasoning we have $(\nu_2)\ti{ac}\sim \nu\ti{ac}$.

Now we need to ensure that this construction lead to the desired singular parts.
By the definition the measures we ensured that on the complement of the interior of $F$ (on the set $\R\backslash(F\backslash \partial F$)) we have the equality of measures
\begin{align*}
\mu_1|\ci{\R\backslash(F\backslash \partial F)} = (\mu_1)\ti{s}|\ci{\R\backslash(F\backslash \partial F)} = \mu|\ci{\R\backslash(F\backslash \partial F)}
\end{align*}
and Lemma \ref{l-dm} implies
\begin{align*}
\mu_2|\ci{\R\backslash(F\backslash \partial F)} \sim (\mu_2)\ti{s}|\ci{\R\backslash(F\backslash \partial F)} \sim \mu|\ci{\R\backslash(F\backslash \partial F)}.
\end{align*}

It remains to check the singular parts on $F\backslash \partial F$.
We begin by recalling that in definition \eqref{d-esupp} the points where the limit-superior is infinite are excluded. So by the definition of $F$ via the intersection of essential supports of the absolutely continuous measures we have that $\mu\ti{s}|\ci {F\backslash \partial F}\equiv 0$. By the definition of $u_2$ on $F\backslash\partial F$, the same is true for $(\mu_2)\ti{s}$. Indeed, for any closed set $X\subset F\backslash \partial F$ there exists an $\e>0$ such that $u_2(x)\in (\e,\pi-\e)$ for all $x\in X$. By equation \eqref{e-arg}, this means that $$\lim_{y\downarrow0} \Im K\nu_2 (x+iy) \neq 0\quad\text{for all }x\in X.$$ In virtue of Theorem \ref{t-AS} (applied to the measures $\mu_\alpha = \mu_2$ and $\mu = \nu_2$) it follows that $(\mu_2)\ti{s}(X)=0$. Whereby the singular parts satisfy the desired property also on $F\backslash \partial F$.
\end{proof}

If the $\{\f_n\}$ form an orthonormal sequence, the following lemma is proved as a corollary to the main theorem in \cite{JakLast2000}. Although, their proof extends immediately to the non-orthogonal case, we decided to include a new shorter proof here.

\begin{lem}\label{l-MSAS}
Assume the hypotheses of Theorem \ref{t-firsttheo} and assume that $\p$ is a product of absolutely continuous measures. Then $(\mu_\omega)\ti s\perp (\mu_\eta)\ti s$ almost surely with respect to the product measure. In particular (with the notation of the proof of Theorem \ref{t-secondtheo}), we have $\mu\ti s\perp \nu\ti s$ almost surely with respect to the product measure. 
\end{lem}

\begin{proof}
Assume that the set $S=\{(\omega,\eta): (\mu_\omega)\ti s\not\perp (\mu_\eta)\ti s \}$ has positive product measure. Because $\p$ is assumed to be a product of absolutely continuous measures, there then exists a pair $(\omega, \eta)\in S$ such that $H_\omega$ is a rank one perturbation of $H_\eta$. But by Aronszajn--Donoghue theory, see Theorem \ref{t-AD}, this is not possible.
\end{proof}

{\bf Acknowledgments.} The author would like to thank Alexei Poltoratski for suggesting the problems which led to this paper as well as for the many insightful discussions and comments along the way. Further thanks to the referees.

\bibliographystyle{amsplain}

\end{document}